\documentclass{amsart}

\usepackage{amsmath,amssymb,url,rotating}
\usepackage[colorlinks=true, urlcolor=rltblue] {hyperref}
\usepackage{color}
\definecolor{rltblue}{rgb}{0,0,0.75}
\usepackage[all]{xy}

\title{A Feiner Look at the Intermediate Degrees}


\author[Hirschfeldt]{Denis R.\ Hirschfeldt}
\author[Kach]{Asher M.\ Kach}
\author[Montalb\'an]{Antonio Montalb\'an}

\address{Department of Mathematics, University of Chicago}
\email{drh@math.uchicago.edu}
\address{Department of Mathematics, University of California,
Berkeley}
\email{antonio@math.berkeley.edu}
\address{Google LLC}
\email{asher.kach@gmail.com}

\thanks{Hirschfeldt was partially supported by NSF grant DMS-1600543.
Montalb\'an was partially supported by NSF grant DMS-1954062.
The authors thank Dan Rosendorf for helping typeset the hierarchy
figure in Proposition~\ref{propHierarchy}.}

\theoremstyle{definition}
\newtheorem{defn}{Definition}[section]
\newcounter{claimCounter}[defn]

\newtheorem{rmk}[defn]{Remark}

\newtheorem{conj}[defn]{Conjecture}

\theoremstyle{plain}
\newtheorem{thm}[defn]{Theorem}
\newtheorem{theorem}[defn]{Theorem}
\newtheorem{prop}[defn]{Proposition}
\newtheorem{lemma}[defn]{Lemma}

\newtheorem{corollary}[defn]{Corollary}

\numberwithin{equation}{section}

\newcommand{\sub}[1]{_{\textrm{\tiny{#1}}}}
\newcommand{\tless}{<\sub{T}}
\newcommand{\tequiv}{\equiv\sub{T}}
\newcommand{\uhr}{\upharpoonright}
\newcommand{\converges}{\mathord{\downarrow}}

\newcommand{\diverges}{\mathord{\uparrow}}
\renewcommand{\leq}{\leqslant}
\renewcommand{\geq}{\geqslant}
\renewcommand{\phi}{\varphi}
\DeclareMathOperator{\dom}{dom}

\newcommand{\w}{\omega}
\newcommand{\zero}[1]{{\emptyset}^{{(#1)}}}
\newcommand{\zeroj}{\emptyset'}

\def\om{\w}
\def\B{{\mathcal B}}

\def\subseteqne{
\hbox{%
\begin{rotate}{45}$\subsetneq$\end{rotate}%
}%
}

\def\subseteqse{
\setbox1=\hbox{\begin{rotate}{-45}$\subsetneq$\end{rotate}}%
\kern -7pt\hbox{\raise 6.72pt\copy1}
}%

\begin{document}

\date{\today}

\begin{abstract}
We say that a set $S$ is $\Delta^0_{(n)}(X)$ if membership of $n$ in
$S$ is a $\Delta^0_{n}(X)$ question, uniformly in $n$.  A set $X$ is
\emph{low for $\Delta$-Feiner} if every set $S$ that is
$\Delta^0_{(n)}(X)$ is also $\Delta^0_{(n)}(\emptyset)$. It is easy to
see that every low$_n$ set is low for $\Delta$-Feiner, but we show
that the converse is not true by constructing an intermediate
c.e.\ set that is low for $\Delta$-Feiner.  We also study variations
on this notion, such as the sets that are $\Delta^0_{(bn+a)}(X)$,
$\Sigma^0_{(bn+a)}(X)$, or $\Pi^0_{(bn+a)}(X)$, and the sets that are
low, intermediate, and high for these classes.  In doing so, we obtain
a result on the computability of Boolean algebras, namely that there
is a Boolean algebra of intermediate c.e.\ degree with no computable
copy.
\end{abstract}

\maketitle

\section{Introduction}
\label{secIntro}

In~\cite{Fein1970}, Feiner introduced a hierarchy of complexities that
we term the \emph{Feiner $\Delta$-hierarchy}, for sets computable
in~$\zero{\w}$. His original motivation was to build a computably
enumerable Boolean algebra that has no computable copy. We can relativize
this hierarchy to an arbitrary set $X$, and define the analogous
\emph{Feiner $\Sigma$-hierarchy} and \emph{Feiner $\Pi$-hierarchy}, as
follows.

\begin{defn}
\label{defnFDelta} 
Fix a set $X \subseteq \w$ and $a,b\in \om$ with $b \geq 1$.  

A set $S \subseteq \w$ is \emph{$\Delta^0_{(bn+a)}(X)$ in the Feiner
$\Delta$-hierarchy}, denoted $S \in \Delta^0_{(bn+a)}(X)$, if there
is a Turing functional $\Phi_e$ such that
\[
        \Phi_e^{X^{(bn+a-1)}}(n) = 
               \begin{cases}  
                   1   &  \mbox{if } n \in S  \\
		   0   &  \mbox{if } n \notin S,   
               \end{cases}
\]
where $X^{(m)}$ is the $m$th Turing jump of $X$ (and
$X^{(-1)}=X^{(0)}=X$).

A set $S \subseteq \w$ is \emph{$\Sigma^0_{(bn+a)}(X)$ in the Feiner
$\Sigma$-hierarchy}, denoted $S \in \Sigma^0_{(bn+a)}(X)$, if there
is a computably enumerable operator $W_e$ such that
\[
    n \in S \iff n \in W_e^{X^{(bn+a-1)}}.
\]

A set $S \subseteq \w$ is \emph{$\Pi^0_{(bn+a)}(X)$ in the Feiner
$\Pi$-hierarchy}, denoted $S \in \Pi^0_{(bn+a)}(X)$, if its
complement is $\Sigma^0_{(bn+a)}(X)$.
\end{defn}

In other words, a set $S$ is $\Delta^0_{(bn+a)}(X)$ (respectively,
$\Sigma^0_{(bn+a)}(X)$ or $\Pi^0_{(bn+a)}(X)$) if membership of $n$
in $S$ is a $\Delta^0_{bn+a}(X)$ (respectively, $\Sigma^0_{bn+a}(X)$
or $\Pi^0_{bn+a}(X)$) question, uniformly in $n$. It is easy to see
that every $\Delta^0_{(bn+a)}(X)$ set is $\Sigma^0_{(bn+a)}(X)$ and
$\Pi^0_{(bn+a)}(X)$, and that every $\Sigma^0_{(bn+a)}(X)$ and
$\Pi^0_{(bn+a)}(X)$ set is computable in $X^{(\w)}$ (see
Proposition~\ref{propHierarchy}).

We study the classes of sets that are low, intermediate, and high with
respect to these complexity classes. Given a relativizable complexity
class $\Gamma$ (for example, $\Delta^0_2$ or $\Delta^0_{(bn+a)}$), we
say that a set $X \subseteq \w$ is \emph{low for $\Gamma$} if every
set $S$ belonging to $\Gamma(X)$ belongs to $\Gamma(\emptyset)$; that
a set $X \subseteq \w$ is \emph{high for $\Gamma$} if every set $S$
belonging to $\Gamma(\zeroj)$ belongs to $\Gamma(X)$; and that a set
$S$ is \emph{intermediate for $\Gamma$} if it is neither low nor high
for $\Gamma$. Thus, for example, a set $X$ is low for $\Delta^0_2$ if
and only if it is low in the usual sense. More generally, the low$_n$
sets are the sets that are low for $\Delta^0_{n+1}$. It is not hard to
see that every low$_n$ set is low for $\Delta^0_{(bn+a)}$ for all $a$
and $b$, for example (see Proposition~\ref{propLowkHighk}).

Since $S \in \Sigma^0_{(an+b)}$ if and only if $\overline{S} \in
\Pi^0_{(an+b)}$, being low for $\Sigma^0_{(an+b)}$ is equivalent to
being low for $\Pi^0_{(an+b)}$, so from now on we do not consider the
latter notion explicitly, and similarly for being high or intermediate
for $\Pi^0_{(an+b)}$.

In Section~\ref{secBasic}, we show that for all $a,b,a',b' \in \w$
with $b,b' \geq 1$, a set is low for $\Delta^0_{(b n + a)}$ if and
only if it is low for $\Delta^0_{(b' n + a')}$.  Consequently, when a
set is low for $\Delta^0_{(bn+a)}$ for some (and hence all) $a,b$, we
say it is \emph{low for $\Delta$-Feiner}. Similar results in that
section justify analogous definitions of \emph{high for
$\Delta$-Feiner}, \emph{low for $\Sigma$-Feiner}, and \emph{high for
$\Sigma$-Feiner}. We also show that being low for $\Sigma$-Feiner
implies being low for $\Delta$-Feiner, and similarly for highness.

\subsection*{Our Results}

In Section~\ref{secTuring}, we prove that there exists a computably
enumerable set of intermediate Turing degree (in the usual sense of
being neither low$_n$ nor high$_n$ for any $n$) that is low for
$\Sigma$-Feiner, and hence low for $\Delta$-Feiner. We therefore find
that (in general and for c.e.\ sets),
\[
\mbox{low}_1 \subsetneq \mbox{low}_2\subsetneq \cdots
\subsetneq\mbox{low}_n\subsetneq \cdots \subsetneq \mbox{low for
  $\Sigma$-Feiner} \subseteq \mbox{low for $\Delta$-Feiner}.
\]
We conjecture that the last containment is also proper.

We also examine the classes of sets that are high for $\Delta$-Feiner
and high for $\Sigma$-Feiner.  We obtain similar results, and in
particular find that (in general and for c.e.\ sets), 
\[
\mbox{high}_1 \subsetneq \mbox{high}_2\subsetneq \cdots
\subsetneq\mbox{high}_n\subsetneq \cdots \subsetneq \mbox{high for
  $\Sigma$-Feiner} \subseteq \mbox{high for $\Delta$-Feiner}.
\]
Again we conjecture that the last containment is proper.

Finally, we show that there is a c.e.\ set that is intermediate for
$\Delta$-Feiner, and hence intermediate for $\Sigma$-Feiner. Thus,
assuming our conjectures above hold, the intermediate (c.e.)\ degrees
can be split into five nonempty classes: low for $\Sigma$-Feiner, low
for $\Delta$-Feiner but not for $\Sigma$-Feiner, intermediate for
$\Delta$-Feiner, high for $\Delta$-Feiner but not for $\Sigma$-Feiner,
and high for $\Sigma$-Feiner.

\subsection*{An Application to Computable Structures}

By extending the ideas in~\cite{Fein1970}, Thurber obtained the
following result.

\begin{theorem} [Thurber~\cite{Thur1994}; see also~{\cite[\S~18.3]{AsKn2000}}] 
\label{thmThur} 
There is a sequence of infinitary sentences $\psi_0, \psi_1, \dots$ in
the language of Boolean algebras such that for every set $S \subseteq
\om$, the following are equivalent.
\begin{enumerate}

\item \label{th:FeinerThurber part 1} There exists a computable
  Boolean algebra $\B$ such that  $S = \{n: \B \vDash \psi_n\}$.

\item The set $S$ is $\Pi^0_{(2n+4)}(\emptyset)$.  

\end{enumerate}
\end{theorem}

\begin{corollary}
\label{boolcor}
If $X$ is not low for $\Sigma$-Feiner, then there is an
$X$-computable Boolean algebra $\B$ that has no computable copy.
\end{corollary}

\begin{proof}
Let $S \in \Pi^0_{(2n+4)}(X) \setminus \Pi^0_{(2n+4)}(\emptyset)$.
Relativizing Theorem~\ref{thmThur} to $X$, there is an $X$-computable
Boolean algebra $\B$ such that $S$ satisfies (\ref{th:FeinerThurber
part 1}) of Theorem~\ref{thmThur}. But then this Boolean algebra
cannot be computable, or else we would have that $S$ is
$\Pi^0_{(2n+4)}(\emptyset)$ by Theorem~\ref{thmThur}.
\end{proof}

The same result follows from the work of Kach in~\cite{Kach2010},
where he studied the complexity of the Ketonen invariants on a certain
class of Boolean algebras: the class of \emph{depth zero} Boolean
algebras. He proved that a depth zero, rank $\w$ Boolean algebra has
a computable copy if and only if its Ketonen invariant is
$\Sigma^0_{(2n+3)}(\emptyset)$. As this result relativizes, it
follows that if a depth zero, rank $\w$ Boolean algebra has a
presentation in a low for $\Sigma$-Feiner set, then it has a
computable copy. We similarly obtain Corollary \ref{boolcor} from
Kach's result.

Our results below (Theorem~\ref{thmHigh} or
Theorem~\ref{thmIntermediate}) show that there exists an intermediate
c.e.\ degree that is not low for $\Sigma$-Feiner, so we obtain the
following corollary, which contrasts with Knight and Stob's result
in~\cite{KnSt2000} that every low$_4$ Boolean algebra has a computable
copy. (Whether every low$_5$ Boolean algebra has a computable copy
remains a well-known open question.)

\begin{corollary}
There is a Boolean algebra of intermediate c.e.\ degree that has no
computable copy.
\end{corollary}

\subsection*{Notation}

Though our notation for the most part follows~\cite{Soar1987}, we
review certain aspects of it briefly. We use upper case Greek letters
(e.g., $\Phi$, $\Psi$, $\Xi$, $\Theta$, etc.)\ to denote Turing
functionals and lower case Greek letters
(e.g., $\phi$, $\psi$, $\xi$, $\theta$, etc.)\ to denote the
corresponding use functions. We use $W_e$ to denote the domain of the
$e$th functional $\Phi_e$ and $( W_e)^{[i]}$ to denote the subset
$\{\langle x,y \rangle \in W_e : y=i\}$. We write $X =^*
Y$ to denote that the symmetric difference $(X \setminus Y) \cup (Y
\setminus X)$ is finite.

\section{Parameter Independence and Other Basic Results}
\label{secBasic}

Before studying which sets are low/intermediate/high for
$\Delta$-Feiner and $\Sigma$-Feiner, we eliminate the need for working
with $\Delta^0_{(bn+a)}$ and $\Sigma^0_{(bn+a)}$ sets for varying $a$
and $b$, through a sequence of quick lemmas. Though we state and prove
these lemmas only for being low for $\Delta^0_{(bn+a)}$, all still
work (with obvious modifications to their proofs) for being high for
$\Delta^0_{(bn+a)}$, low for $\Sigma^0_{(bn+a)}$, and high for
$\Sigma^0_{(bn+a)}$. Throughout this section, $b \geq 1$.

\begin{lemma}
\label{lemTranslate}
If $X$ is low for $\Delta^0_{(bn+a)}$, then $X$ is low for
$\Delta^0_{(bn+a')}$.
\end{lemma}

\begin{proof}
Suppose that $X$ is low for $\Delta^0_{(bn+a)}$. Let $S \in
\Delta^0_{(bn+a')}(X)$. Then $\pi(S) := \{x \in \omega : x+a-a' \in
S\} \in \Delta^0_{(bn+a)}(X)$, so $\pi(S) \in
\Delta^0_{(bn+a)}(\emptyset)$. It follows that $S \in
\Delta^0_{(bn+a')}(\emptyset)$.
\end{proof}

\begin{lemma}
\label{lemShrink}
If $X$ is low for $\Delta^0_{(bn)}$, then $X$ is low for
$\Delta^0_{(n)}$.
\end{lemma}

\begin{proof}
Suppose that $X$ is low for $\Delta^0_{(bn)}$. Let $S \in
\Delta^0_{(n)}(X)$. Define sets $S_i$ for $i < b$ by $S_i := \{x \in
\omega : bx + i \in S\}$. Then $S_i \in \Delta^0_{(bn+i)}(X)$. By
Lemma~\ref{lemTranslate}, $X$ is low for $\Delta^0_{(bn+i)}$, so $S_i
\in \Delta^0_{(bn+i)}(\emptyset)$. It follows that $S \in
\Delta^0_{(n)}(\emptyset)$.
\end{proof}

\begin{lemma}
\label{lemStretch}
If $X$ is low for $\Delta^0_{(n)}$, then $X$ is low for
$\Delta^0_{(bn)}$.
\end{lemma}

\begin{proof}
Suppose that $X$ is low for $\Delta^0_{(n)}$. Let $S \in
\Delta^0_{(bn)}(X)$. Then $\pi(S) := \{bx : x \in S\} \in
\Delta^0_{(n)}(X)$, and hence $\pi(S) \in
\Delta^0_{(n)}(\emptyset)$. It follows that $S \in
\Delta^0_{(bn)}(\emptyset)$.
\end{proof}

Combining Lemmas~\ref{lemTranslate},~\ref{lemShrink},
and~\ref{lemStretch}, we obtain the desired invariance.

\begin{prop}
\label{propInvariance}
If $X$ is low (respectively, intermediate or high) for
$\Delta^0_{(bn+a)}$ for some $a,b$, then $X$ is low (respectively,
intermediate or high) for $\Delta^0_{(b'n+a')}$ for all $a',b'$. If
$X$ is low (respectively, intermediate or high) for
$\Sigma^0_{(bn+a)}$ for some $a,b$, then $X$ is low (respectively,
intermediate or high) for $\Sigma^0_{(b'n+a')}$ for all $a',b'$. 
\end{prop}

This proposition justifies our use of terms like low for
$\Delta$-Feiner. We also have the following relationship.

\begin{prop}
\label{propSigmaDelta}
If $X$ is low (respectively, high) for $\Sigma$-Feiner, then $X$ is
low (respectively, high) for $\Delta$-Feiner.
\end{prop}

\begin{proof}
Suppose $X$ is low for $\Sigma$-Feiner. Let $S \in
\Delta^0_{(n)}(X)$. Then $S \in \Sigma^0_{(n)}(X)$ and $S \in
\Pi^0_{(n)}(X)$, as $\Delta^0_{(n)}(X) \subset
\Sigma^0_{(n)}(X),\Pi^0_{(n)}(X)$. Since $X$ is low for
$\Sigma^0_{(n)}$, and hence low for $\Pi^0_{(n)}$, we have $S \in
\Sigma^0_{(n)}(\emptyset)$ and $S \in \Pi^0_{(n)}(\emptyset)$. It
follows that $S \in \Delta^0_{(n)}(\emptyset)$.

The proof for highness is analogous.
\end{proof}

As noted above, we conjecture that the converse to this proposition
does not hold.

We finish by noting the relationships between the classes of sets that
are $\Delta^0_{(bn+a)}$, $\Sigma^0_{(bn+a)}$, and $\Pi^0_{(bn+a)}$ for
varying $a,b$. These relationships essentially follow
from Feiner's work in~\cite{Fein1970}; see~\cite{AsKn2000}
or~\cite{Thur1994}, for example.

\begin{prop}
\label{propHierarchy} 
The classes of sets that are $\Delta^0_{(bn+a)}$, $\Sigma^0_{(bn+a)}$, and
$\Pi^0_{(bn+a)}$ satisfy the inclusions
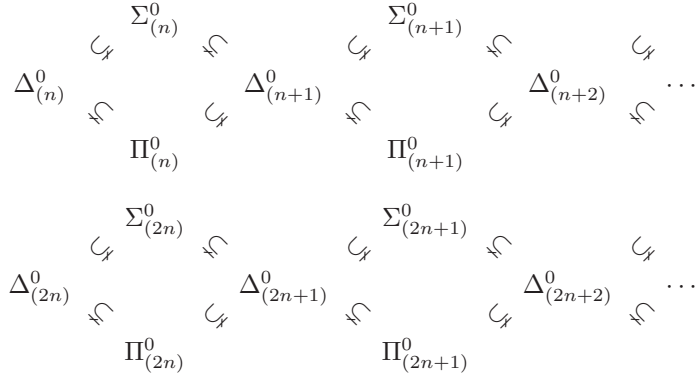
\begin{eqnarray*}
    \begin{array}{c c c c c c c c c c c c c c c c c}
         & & \Sigma^0_{(n)} & & & & \Sigma^0_{(n+1)} & & \\

         & \subseteqne & & \subseteqse & & \subseteqne & & \subseteqse  
                 & & \subseteqne\\

         \Delta^0_{(n)} & & & & \Delta^0_{(n+1)} & & & &
         \Delta^0_{(n+2)}  & & \cdots \\

        & \subseteqse & & \subseteqne & & \subseteqse & & \subseteqne 
                 & & \subseteqse\\

        & & \Pi^0_{(n)} & & & & \Pi^0_{(n+1)} \\

      \\

        & & \Sigma^0_{(2n)} & & & & \Sigma^0_{(2n+1)} \\

        & \subseteqne & & \subseteqse & & \subseteqne & & \subseteqse
                  & & \subseteqne \\

        \Delta^0_{(2n)} & & & & \Delta^0_{(2n+1)} & & & & 
                                       \Delta^0_{(2n+2)} & & \cdots \\

        & \subseteqse & & \subseteqne & & \subseteqse & & \subseteqne 
                  & & \subseteqse \\

        & & \Pi^0_{(2n)} & & & & \Pi^0_{(2n+1)}\\
    \end{array}
\end{eqnarray*}
\end{prop}

In particular, a set cannot be both low for $\Delta$-Feiner and high
for $\Delta$-Feiner, and similarly for $\Sigma$-Feiner.

\section{The Intermediate Turing Degrees}
\label{secTuring}

We now turn our attention to studying which Turing degrees can be low,
intermediate, or high for $\Delta$-Feiner and which can be low,
intermediate, or high for $\Sigma$-Feiner. The following proposition
was essentially noted in~\cite{Kach2010}.

\begin{prop}
\label{propLowkHighk} 
Every low$_k$ Turing degree is low for $\Delta$-Feiner and low for
$\Sigma$-Feiner. Every high$_k$ Turing degree is high for
$\Delta$-Feiner and high for $\Sigma$-Feiner.
\end{prop}

\begin{proof}
Let $X$ be a low$_k$ set. Provided $n \geq k$, we have
\begin{eqnarray*}
X^{(n)} = \left( X^{(k)} \right)^{(n-k)} 
       \tequiv \left( \emptyset^{(k)} \right)^{(n-k)} 
       = \emptyset^{(n)},
\end{eqnarray*}
and the Turing reductions are uniform in $n$.  It follows that $X$
is low for $\Delta^0_{(n)}$ and low for $\Sigma^0_{(n)}$, and thus
low for $\Delta$-Feiner and low for $\Sigma$-Feiner.
  
If instead $X$ is a high$_k$ set, the equivalence becomes $X^{(n)}
\tequiv \emptyset^{(n+1)}$ for $n \geq k$.  It follows that $X$ is
high for $\Delta^0_{(n)}$ and high for $\Sigma^0_{(n)}$, and thus
high for $\Delta$-Feiner and high for $\Sigma$-Feiner.
\end{proof}

The intermediate degrees have greater complexity: some are low for
$\Delta$-Feiner and $\Sigma$-Feiner, some are intermediate for
$\Delta$-Feiner and $\Sigma$-Feiner, some are high for $\Delta$-Feiner
and $\Sigma$-Feiner, and we conjecture that some behave differently
for $\Delta$-Feiner and for $\Sigma$-Feiner. In order to construct
examples, we will modify the construction of an intermediate
c.e.\ degree. We note the following properties of this construction as
given in~\cite{Soar1987}.

\begin{rmk}
\label{rmkInt} 
Let $q(x)$ be a total computable function as in the proof of Corollary
VIII.3.5 in~\cite{Soar1987}, i.e., a total computable function
satisfying
\begin{equation}
\label{eqnJump}
    Y \tless W^Y_{q(x)} \tless Y' \qquad \text{and} \qquad
    (W^Y_{q(x)})' \tequiv (W_x^{Y'}) \oplus Y'
\end{equation}
for all $x$ and $Y$. We can ensure that $(W^Y_{q(x)})^{[0]}$ is equal
to $Y \times \{0\}$ for all $x$ and $Y$. As noted in that proof, if we
take a fixed point $m$ such that $W^Y_{q(m)}=W^Y_m$ for all $Y$, which
exists by the relativized form of the Recursion Theorem, then the
degree of $W^{\emptyset}_m$ is intermediate. This fact relies only on
the properties in (\ref{eqnJump}), so we will be able to produce
intermediate degrees with additional properties by modifying $q$ while
preserving these properties.

It is easy to check from the proof of the Jump Theorem
in~\cite{Soar1987} (Theorem~VIII.3.1) that the second equivalence in
(\ref{eqnJump}) is uniform in $Y$; i.e., for each $x$ there are
functionals $\Psi$ and $\Xi$ such that $(W^Y_{q(x)})' =
\Psi^{(W_x^{Y'}) \oplus Y'}$ and $(W_x^{Y'}) \oplus Y' =
\Xi^{(W^Y_{q(x)})'}$ for all $Y$. It also follows from that proof that
we can define $\Xi$ so that for all $Y$ and $X =^* W^Y_{q(x)}$, we
have $(W_x^{Y'}) \oplus Y' = \Xi^{X'}$.
\end{rmk}

We will also need the following lemma.

\begin{lemma}
\label{lemZeron} 
There is a partial computable function $f : 2^\omega \rightarrow
\omega$ such that $f(\emptyset^{(n)})=n$ for all $n \geq 0$.
\end{lemma}

\begin{proof}
Let $\{e_i\}_{i \in \w}$ be a computable sequence of integers such
that $\Phi_{e_0}^X(e_0)\converges$ for all $X$ and
$\Phi_{e_{i+1}}^X(e_{i+1})\converges$ if and only if $e_i \in
X$. Then $e_i \in \emptyset^{(n)}$ if and only if $i < n$, so we can define
$f(X)$ to be the least $i$, if any, such that $e_i \notin X$.
\end{proof}

We begin by showing the existence of an intermediate c.e.\ degree that
is low for $\Sigma$-Feiner, and hence for $\Delta$-Feiner.

\begin{thm}
\label{thmLow} 
There is a computably enumerable set $A$ of intermediate Turing
degree such that $A$ is low for $\Sigma$-Feiner.
\end{thm}

\begin{proof}
Let $q$ be as in Remark~\ref{rmkInt}, and let $f$ be as in
Lemma~\ref{lemZeron}. We will define a total computable function $r$
such that $W^Y_{r(x)} =^* W^Y_{q(x)}$ whenever $f(Y)$ is defined,
which suffices to ensure that if $m$ is a fixed point of $r$, then the
degree of $W^{\emptyset}_m$ is intermediate. We will also have
$(W^Y_{r(x)})^{[0]} = (W^Y_{q(x)})^{[0]}$ (which recall is equal to $Y
\times \{0\}$). It thus makes sense to establish the convention that
$\Phi_e^{W^Y_{r(x),t}}(k)[t]\converges$ means that this computation
converges and
\[
\left( W^Y_{r(x),t} \right)^{[0]} \uhr \phi_e^{W^Y_{r(x),t}}(k) =
\left( Y \uhr \phi_e^{W^Y_{r(x),t}}(k) \right) \times \{0\}.
\]

Define $r$ so that on oracle $Y$:
\begin{itemize}
  
\item[(a)] $(W^Y_{r(x)})^{[0]} = (W^Y_{q(x)})^{[0]}$.

\item[(b)]  If $f(Y)\diverges$ and $i \notin \omega^{[0]}$  then $i
\notin W^Y_{r(x)}$.

\item[(c)] If $f(Y)\converges = n$ and $i \notin \omega^{[0]}$ enters
$W^Y_{q(x)}$ at stage $t$, then $i \in W^Y_{r(x)}$ unless
$\Phi_e^{W^Y_{r(x),t}}(n)[t]\converges$ and $i <
\phi_e^{W^Y_{r(x),t}}(n)$ for some $e \leq n$.

\item[(d)] No other numbers are in $W^Y_{r(x)}$.

\end{itemize}

The point of item (c) is that (given our convention above) it ensures
that if $e \leq n=f(Y)$ and $\Phi_e^{W^Y_{r(x)}}(n)\converges$, then
$\Phi_e^{W^Y_{r(x)}}(n) = \Phi_e^{W^Y_{r(x),t}}(n)[t]$ for the least
$t$ such that $\Phi_e^{W^Y_{r(x),t}}(n)[t]\converges$. Thus
$\Phi_e^{W^Y_{r(x)}}(n)\converges$ if and only if there is a $t$ such
that $\Phi_e^{W^Y_{r(x),t}}(n)[t]\converges$, which is a
$Y$-c.e.\ condition.

We have $W^Y_{r(x)} =^* W^Y_{q(x)}$ whenever $f(Y)$ is defined, so $r$
satisfies (\ref{eqnJump}) for all such $Y$. Thus, if we let $m$ be a
fixed point of $r$ and let $A=W_m^\emptyset$, then $A$ has
intermediate degree as in  Corollary VIII.3.5 of~\cite{Soar1987}. We
are left with showing that $A$ is low for $\Sigma$-Feiner. We do so by
showing that there is a uniform procedure for computing $A^{(n)}$ from
$W_m^{\emptyset^{(n)}}$, and then applying the previous paragraph.

If $f(Y)$ is defined then the difference between $W^Y_{r(x)}$ and
$W^Y_{q(x)}$ can be computed uniformly from $Y'$, so for each $x$
there is a functional $\Theta$ (defined using the functional $\Psi$ in
Remark~\ref{rmkInt}) for which $(W^Y_{r(x)})' = \Theta^{(W_x^{Y'})
\oplus Y'}$ for all $Y$ such that $f(Y)$ is defined. Since $Y'$ is
encoded into the $0$th column of $W_m^{Y'}=W_{r(m)}^{Y'}$, taking
$x=m$ we actually have a functional $\Theta$ such that  $(W^Y_m)'
= \Theta^{(W_m^{Y'})}$ for all $Y$ such that $f(Y)$ is defined.

Thus $A' = (W_m^\emptyset)' = \Theta^{(W_m^{\emptyset'})}$.
Similarly, $A'' = (W_m^\emptyset)'' = (\Theta^{(W_m^{\emptyset'})})'$,
which is computable from $(W_m^{\emptyset'})' =
\Theta^{(W_m^{\emptyset''})}$ via a reduction that can be found
uniformly from an index for $\Theta$. Continuing in this manner, we
obtain a uniform procedure for computing $A^{(n)}$ from
$W_m^{\emptyset^{(n)}}$.

Now let $S \in \Sigma^0_{(n)}(A)$. Then, by the above, there is a
functional $\Phi_e$ such that $n \in S$ if and only if
$\Phi_e^{W_m^{\emptyset^{(n)}}}(n)\converges$. For $n \geq e$, we can
use the oracle $\emptyset^{(n)}$ to search for a $t$ such that
$\Phi_e^{W_{m,t}^{\emptyset^{(n)}}}(n)[t]\converges$, enumerating $n$
into a set $\widehat{S}$ when such a $t$ is found. By construction,
$\Phi_e^{W_m^{\emptyset^{(n)}}}(n)\converges$ if and only if there is
such a $t$, so $\widehat{S} =^* S$. Since $\widehat{S}$ is in
$\Sigma^0_{(n)}(\emptyset)$, so is $S$. Thus $A$ is low for
$\Sigma$-Feiner.
\end{proof}

We now show the existence of an intermediate c.e.\ degree that is high
for $\Sigma$-Feiner, and hence for $\Delta$-Feiner.

\begin{thm}
\label{thmHigh}
There is a computably enumerable set $A$ of intermediate Turing
degree such that $A$ is high for $\Sigma$-Feiner.
\end{thm}

\begin{proof}
Again let $q$ be as in Remark~\ref{rmkInt} and let $f$ be as in
Lemma~\ref{lemZeron}.

It is not difficult to see that there is a c.e.\ operator $V$ so that
on oracle $Y$:
\begin{itemize}

\item[(a)] If $f(Y)\diverges$, then $V^Y = \emptyset$.

\item[(b)] If $f(Y)\converges = n$, then:
\begin{itemize}

\item If $e > n$, then the $e$th column of $V^Y$ is empty.

\item If $e \leq n$ and $\Phi_e^{Y'}(n)\diverges$, then the $e$th
column of $V^Y$ is $\omega \times \{e\}$.

\item Otherwise, the $e$th column of $V^Y$ is $[0,t] \times \{e\}$,
where $t$ is minimal such that $\Phi_e^{Y_t'}(n)[t]\converges$ and
$Y_t' \uhr \phi_e^{Y_t'}(n) = Y' \uhr \phi_e^{Y_t'}(n)$.

\end{itemize}
\end{itemize}
Here we are thinking of a standard $Y$-enumeration of $Y'$.

Let $s : \w \to \w$ be a total computable function such that, on
oracle $Y$,
\begin{eqnarray*}
W^Y_{s(x)} =
\begin{cases}
\emptyset            & \mbox{if $f(Y)\diverges$} \\
W^Y_{q(x)} \oplus V^Y  & \mbox{otherwise}.
\end{cases}
\end{eqnarray*}
The definition of $s$ ensures that if $f(Y)$ is defined, then
$W^Y_{s(x)} \tequiv W^Y_{q(x)}$, so (\ref{eqnJump}) holds for $s$ in
place of $q$.  Moreover, for each $x$ there is a single functional
$\Theta$ (defined using the functional $\Xi$ in Remark~\ref{rmkInt})
such that $W_x^{Y'} = \Theta^{(W^Y_{s(x)})'}$. Furthermore, for any
$x$ and $e$, we can compute $\Phi_e^{Y'}(n)$ (in the sense of
computing a partial function) uniformly from $W^{Y}_{s(x)}$ for $Y$
such that $f(Y) \converges = n$.

Let $m$ be a fixed point of $s$, let $A=W_m^\emptyset$, and let
$\Theta$ be as above with $x=m$. Then $A$ is intermediate as in
Corollary VIII.3.5 of~\cite{Soar1987}. It follows from the definition
of $m$, $A$, and $\Theta$ that $W_m^{\emptyset'} =
\Theta^{(W_m^\emptyset)'} = \Theta^{A'}$. Similarly,
$(W_m^{\emptyset'})'$ can be obtained from $A'' = (W_m^\emptyset)''$
via a reduction that can be found from an index for $\Theta$, and
hence so can $W_m^{\emptyset''} =
\Theta^{(W_m^{\emptyset'})'}$. Continuing in this manner, we obtain a
uniform procedure for computing $W^{\emptyset^{(n)}}_m$ from
$A^{(n)}$.

Now suppose that $W_e = \dom \Phi_e$ witnesses that $S \in
\Sigma^0_{(n)}(\emptyset')$.  Then there is a uniform procedure for
computing $\Phi_e^{\emptyset^{(n+1)}}(n)$ from
$W^{\emptyset^{(n)}}_m$, and hence from $A^{(n)}$.  Thus $S \in
\Sigma^0_{(n)}(A)$, and so $A$ is high for $\Sigma$-Feiner.
\end{proof}

We next show the existence of a c.e.\ degree that is intermediate for
$\Delta$-Feiner, and hence for $\Sigma$-Feiner (and hence is also
intermediate in the usual sense).

\begin{thm}
\label{thmIntermediate} 
There is a computably enumerable set that is intermediate for
$\Delta$-Feiner.
\end{thm}

\begin{proof}
We will define sets $A_0$ and $A_1$ such that $\Delta^0_{(n)}(A_0)
\nsubseteq \Delta^0_{(n)}(A_1)$ and $\Delta^0_{(n)}(A_1) \nsubseteq
\Delta^0_{(n)}(A_0)$. Then both $A_0$ and $A_1$ must be intermediate
for $\Delta$-Feiner.
  
Let $q$ be as in Remark~\ref{rmkInt}, let $f$ be as in
Lemma~\ref{lemZeron}, and let $r$ be as in the proof of
Theorem~\ref{thmLow}. Let $V$ be a c.e.\ operator so that on oracle
$Y$: 
\begin{itemize}

\item[(a)] If $f(Y)\diverges$, then $V^Y=\emptyset$.

\item[(b)] If $f(Y)\converges = n$, then if
$\Phi_e^{W^Y_{r(x),t-1}}(n)[t-1]\diverges$ and
$\Phi_e^{W^Y_{r(x),t}}(n)[t]\converges$ (under the same convention
as in the proof of Theorem~\ref{thmLow}) for some $e \leq n$, then
all numbers less than $t$ are enumerated into $V^Y$.

\end{itemize}
Note that if $f(Y)\converges=n$ then $V^Y$ is finite (by item (c) in
the definition of $r$), and for the least $t \notin V^Y$ and every $e
\leq n$, we have
$\Phi_e^{W^Y_{r(x)}}(n)=\Phi_e^{W^Y_{r(x),t}}(n)[t]$. Let $s : \w \to
\w$ be a total computable function such that $W^Y_{s(x)} = W^Y_{q(x)}
\oplus V^Y$ for all $Y$.

For $i \in \{ 0,1 \}$, let $p_i : \w \to \w$ be a total computable
function such that
\begin{eqnarray*}
W^Y_{p_i(x)} =
\begin{cases}
\emptyset   & \mbox{if $f(Y)\diverges$} \\
W^Y_{s(x)}    & \mbox{if $f(Y)\converges = i \bmod 2$} \\
W^Y_{r(x)}    & \mbox{otherwise}.
\end{cases}
\end{eqnarray*}
The definition of the $p_i$ ensures that if $f(Y)$ is defined, then
$W^Y_{p_i(x)} \tequiv W^Y_{q(x)}$, the set $Y$ is effectively coded in
$W^Y_{p_i(x)}$, and the difference between $W^Y_{p_i(x)}$ and
$W^Y_{q(x)}$ is uniformly computable in $Y'$. Thus, for each $x$ and
$i$, there is a functional $\Theta$ such that $(W^Y_{p_i(x)})' =
\Theta^{W_x^{Y'} \oplus Y'}$ for all $Y$ such that $f(Y)$ is
defined. Furthermore, if $f(y)\converges \neq i \bmod 2$, then
$W^Y_{p_i(x)} =^* W^Y_{q(x)}$, and if $f(Y) \converges = i \bmod 2$,
then $W^Y_{p_i(x)}$ codes $W^Y_{q(x)}$ in its even bits. So, for each
$x$ and $i$, there are functionals $\Gamma$ and $\Lambda$ such that
$W_x^{Y'} = \Gamma^{(W^Y_{p_i(x)})'}$ if $f(Y) = i \bmod 2$ and
$W_x^{Y'} = \Lambda^{(W^{Y}_{p_i(x)})'}$ otherwise, for all $Y$ such
that $f(Y)$ is defined.

Let $m_i$ be a fixed point of $p_i$ and let
$A_i=W_{m_i}^\emptyset$. By the previous paragraph, and arguing as in
the previous two proofs, there are uniform procedures for computing
$A_i^{(n)}$ from $W_{m_i}^{\emptyset^{(n)}}$ and vice-versa. As
mentioned above, it suffices to show that $\Delta^0_{(n)}(A_0)
\nsubseteq \Delta^0_{(n)}(A_1)$ and $\Delta^0_{(n)}(A_1) \nsubseteq
\Delta^0_{(n)}(A_0)$.

Towards a contradiction, assume that $\Delta^0_{(n)}(A_0) \subseteq
\Delta^0_{(n)}(A_1)$.  Define a total function $h : \w \to 2$ as
follows. If $n$ is odd, then $h(n)=0$. Otherwise, let $t$ be least
such that $2t+1 \notin W^{\emptyset^{(n)}}_{m_0}$ (which must exist by
the definition of the operator $V$), let $e=n/2$, and compute
$\Phi_e^{W_{r(m_1),t}^{\emptyset^{(n)}}}(n)[t]$. If this value is
defined, then let $h(n)$ be different from it; otherwise let $h(n)=0$.

The value $h(n)$ can be computed uniformly from
$W^{\emptyset^{(n)}}_{m_0}$, and hence from $A_0^{(n)}$, so $h \in
\Delta^0_{(n)}(A_0)$.  Thus, by assumption, $h \in
\Delta^0_{(n)}(A_1)$.  Consequently, there is a functional $\Phi_e$
such that $\Phi_e^{W_{p_1(m_1)}^{\emptyset^{(n)}}}(n) =
\Phi_e^{W_{m_1}^{\emptyset^{(n)}}}(n) = h(n)$ for all $n$. Let
$n=2e$. Then, by the definition of $r$ and $p$, for the least $t$ such
that $2t+1 \notin W^{\emptyset^{(n)}}_{m_0}$, we have
$\Phi_e^{W_{p_1(m_1)}^{\emptyset^{(n)}}}(n) = \Phi_e^{W_{r(m_1)}^{\emptyset^{(n)}}}(n) =
\Phi_e^{W_{r(m_1),t}^{\emptyset^{(n)}}}(n)[t]$, so $h(n) \neq
\Phi_e^{W_{p_1(m_1)}^{\emptyset^{(n)}}}(n)$, yielding a contradiction.

The symmetric argument shows that $\Delta^0_{(n)}(A_1) \nsubseteq
\Delta^0_{(n)}(A_0)$. It follows that both $A_0$ and $A_1$ are
intermediate for $\Delta$-Feiner.
\end{proof}

We finish with the following conjectures.

\begin{conj}
There is a computably enumerable set that is intermediate for
$\Sigma$-Feiner but low for $\Delta$-Feiner. 
\end{conj}

\begin{conj}
There is a computably enumerable set that is intermediate for
$\Sigma$-Feiner but high for $\Delta$-Feiner.
\end{conj}

\end{document}